\documentclass[11pt]{article}
\usepackage{amsthm,amsfonts, amsbsy, amssymb,amsmath,graphicx}
\usepackage{graphics}
\usepackage[english]{babel}
\usepackage{graphicx}
\usepackage{lineno}
\usepackage{enumerate}
\usepackage{tikz}
\usepackage{tkz-graph}
\usepackage{tkz-berge}
\usepackage{hyperref}
\usepackage{color}
\usepackage{tikz}

\hypersetup{colorlinks=true}

\hypersetup{colorlinks=true, linkcolor=blue, citecolor=blue,urlcolor=blue}

\input epsf
\usepackage{epic,latexsym,amssymb}
\usepackage{tikz}

\usepackage{tikz,epsf}

\usepackage{amsfonts}
\usepackage{amscd}
\usepackage{amsmath}
\usepackage{graphicx}
\usepackage{color}
\usepackage{caption,subcaption}

\newtheorem{theorem}{Theorem}

\newtheorem{lemma}[theorem]{Lemma}

\tikzstyle{vertex}=[circle, draw, inner sep=0pt, minimum size=6pt]

\newcommand{\QEDmark}{\mbox{\textsc{qed}}}
\newcommand{\proofStarter}[1]{\textsc{#1} }
\begin{document}

\title{On the upper Bound of double Roman dominating function}
\date{}
\author{ {\small A. Teimourzadeh$^1$, D.A. Mojdeh$^2$\thanks{Corresponding author}}\\
 {\small$^{1,2}$Department of Mathematics, University of Mazandaran,}\\
  {\small Babolsar, Iran}\\
{\tt $^1$atiehteymourzadeh@gmail.com}\\
 {\tt$^2$damojdeh@umz.ac.ir}}
\maketitle

\begin{abstract}
A double Roman Dominating function on a graph $G$ is a function $ f:V\rightarrow \{0,1,2,3\}$ such that the following conditions hold.
If $f(v)=0$, then vertex $v$ must have at least two neighbors in $V_2$ or one neighbor in $V_3$ and
if $f(v)=1$, then vertex $v$ must have at least one neighbor in $V_2\bigcup V_3$.
The weight of a double Roman dominating function is the sum $w_f=\sum_{v\in V(G)}{f(v)}$. In this paper, we improve the upper bounds
of $\gamma_{dR}(G)$ that has already obtained and we show that $\gamma_{dR}(G)\leq\dfrac{12n}{11}$, for any graph with $\delta(G) \ge 2$. This bound improve the bounds
that have already been presented in \cite{chen} and \cite{kkcs}. Finally we prove the conjecture posed in \cite{kkcs}.
\end{abstract}
\noindent\textbf{Keywords:} Double Roman domination, upper bound.
\newline\textbf{MSC 2010}: 05C69

\section{Introduction}
Let $G=(V,E)$ be a graph of order $n$  with $V=V(G)$
and $E=E(G)$. The open neighborhood of a vertex $v\in V(G)$ is the set $N(v)=\{u: uv \in  E(G)\}$. The closed
neighborhood of a vertex $v\in V(G)$  is $N[v]=N(v)\cup \{v\}$. The open neighborhood of a set $S\subseteq V$ is the
set $N(S)=\cup_{v\in S}N(v)$. The closed neighborhood of a set $S\subseteq V$ is the set
$N[S]=N(S)\cup S=\cup_{v\in S}N[v]$. We denote the degree of $v$ by $d_G (v)=|N(v)|$. Given a set $S\subseteq V$,
the private neighborhood $pn[v,S]$ of $v\in S$ is defined by $ pn[v,S]=N[v]-N[S-\{v\}]$, equivalently,
$ pn[v,S]=\{u\in V: N[u]\cap S=\{v\}\}$. Each vertex in $pn[v,S]$ is called a private neighbor of $v$. By
$\Delta = \Delta(G)$ and $\delta = \delta(G)$, we denote the maximum degree and minimum degree
of a graph $G$, respectively.  We write $K_n$, $P_n$ and $C_n$ for the complete graph, path and cycle of order $n$,
respectively. A tree $T$ is an acyclic connected graph.

A set $S\subseteq V$ in a graph $G$ is called a dominating set
if $N[S]=V$. The domination number $\gamma (G)$ of $G$ is the minimum cardinality
of a dominating set in $G$, and a dominating set of $G$ of cardinality
$\gamma(G)$ is called a $\gamma$-set of $G$.
A subset $S \subseteq V$ is a $k$-dominating set if every vertex of $V - S$ has at least $k$ neighbors in $S$.
The $k$-domination number $\gamma_k(G)$ is the minimum cardinality of a $k$-dominating set of $G$ (see \cite{hhs}).

Given a graph $G$ and a positive integer $m$, assume that $g : V(G) \rightarrow \{0, 1, 2, \ldots , m\}$ is a function,
and suppose that $(V_0, V_1, V_2,\ldots, V_m)$ is the ordered partition of $V$ induced by $g$, where
$V_i = \{v \in V : g (v) = i\}$ for $i \in \{0, 1, \ldots , m\}$.  So we
can write $g = (V_0, V_1, V_2,\ldots, V_m)$. A Roman dominating function on graph $G$ is a function
$ f:V\rightarrow \{0,1,2\}$ such that if $v \in V_0$ for some $v\in V$, then there  exists a vertex $w \in N(v)$  with
$w \in V_2$. The weight of a Roman dominating function is the sum $w_f=\sum_{v\in V(G)}{f(v)}$, and the minimum weight
of $w_f$ for every Roman dominating function $f$ on $G$ is called the Roman domination number of $G$, denoted by
$\gamma_{R} (G)$.

Roman domination was introduced by Cockayne et al. in \cite{cdhh}, although this notion
 was inspired by the work of ReVelle et al in  \cite{rer}, and Stewart in \cite{s}, although in the present  several papers have been  issued on  Roman domination, for example \cite{chhm, cdhh, hrsw, h, lc, moj3}.
 The original study of Roman domination was motivated by the defense strategies used to defend the Roman Empire during the reign of Emperor Constantine the Great, 274-337 A.D. He decreed that for all cities in the Roman Empire, at most two legions should be stationed. Further, if a location having no legions was attacked, then it must be within the vicinity of at least one city at which two legions were stationed, so that one of the two legions could be sent to defend the attacked city. This part of history of the Roman Empire gave rise to the mathematical concept of Roman domination, as originally defined and discussed by Stewart  \cite{s} in (1999), and ReVelle and Rosing  \cite{rer} in (2000).

Beeler et al. \cite{bhh} have  defined double Roman domination on 2016.  What they propose is a stronger
version of Roman domination that doubles the protection by ensuring that any attack can be defended by at least two legions.
In  Roman domination  at most two Roman legions are deployed at any one location. But
as we will see in what follows, the ability to deploy three legions at a given location provides a level of defense that is both
stronger and more flexible, at less than the anticipated additional cost.\\
A double Roman Dominating function on a graph $G$ is a function $ f:V\rightarrow \{0,1,2,3\}$ such that the following conditions are met:

\noindent(a) if $f(v)=0$, then vertex $v$ must have at least two neighbors in $V_2$ or one neighbor in $V_3$.\\
(b) if $f(v)=1$ , then vertex $v$ must have at least one neighbor in $V_2\bigcup V_3$.

The weight of a double Roman dominating function is the sum $w_f=\sum_{v\in V(G)}{f(v)}$, and the minimum weight of $w_f$ for every double Roman dominating function $f$ on $G$ is called double roman domination number of $G$. We denote this number with $\gamma_{dR} (G)$ and
a double Roman dominating function of $G$ with weight $\gamma_{dR}(G)$ is called a $\gamma_{dR}(G)$-function of $G$, although in the present  several papers have issued on double Roman domination, for example \cite{acs, amjadi, bhh, jr, kkcs, moj1,moj2}.\\

\noindent\textbf{Proposition A}\label{pro A}
 (Beeler et al.2016 \cite{bhh}) In double Roman dominating function of weight
$\gamma_{dR}(G)$, no vertex needs to be assigned the value 1.\\

By Proposition A, in what follows, when we consider a $\gamma_{dR}$-function of a graph $G$ we assume no vertex assigned the value 1.\\

\noindent\textbf{Proposition B}\label{pro B}
(J. Rad et al. \cite{jr}) For any connected graph $G$ of order $n$ with maximum degree $\Delta$, $\gamma_{dR}(G)\leq2n-2\Delta+1$.\\

\noindent\textbf{Proposition C}\label{pro C} (A.H. Ahangar et al. \cite{acs}). If $T$ is a spider of order $n\geq3$, then $\gamma_{dR}(T)\leq n+1.$\\
\textbf{Proposition D}\label{pro D}(R. Khoeilar et al. \cite{kkcs}). Let $G$ be a simple graph of order $n\geq5$, $\delta(G) \ge 2$ and with no component isomorphic to $C_{5}$ or $C_{7}$. Then
$\gamma_{dR}(G) \leq\frac{11n}{10}$.\\

\noindent\textbf{Proposition E}\label{pro E}( \cite{chen}).  Let $G$ be a simple graph of order $n\geq5$, $\delta(G) \ge 2$ and with no component isomorphic to $C_{5}$. Then $\gamma_{dR}(G) \leq \lfloor\frac{13n}{11}\rfloor$.\\

In this paper, we improve the upper bounds 
of $\gamma_{dR}(G)$ that have already been presented in Propositions \ref{pro D}  and \ref{pro E} by showing that $\gamma_{dR}(G)\leq\dfrac{12n}{11}$, for any graph with $\delta(G) \ge 2$. Finally we prove the conjecture posed in \cite{kkcs}.

\section{Main results}
Before presenting the proof of main result, we give some lemmas that are useful for investigation. For integers $m$ and $k$ where $m\geq3$ and $k\geq1$, let $C_{m,k}$ be the graph obtained from a cycle $C_{m}:x_{1}x_{2}\cdots x_{m}x_{1}$ and a path $y_{1}y_{2}\cdots y_{k}$ by adding the edge $x_{1}y_{1}$. Let $\mathcal{H}$ be the family of all connected graphs $G$ with $\delta(G)\geq2$ and $\gamma_{dR}(G)\leq\frac{12n}{11}$. Let $Q$ be a graph obtained from two cycles $C_{5}$ and $C'_{5}$, by joining a vertex of $C_{5}$ to exactly one vertex of $C'_{5}$. Let $G_{Q}$ be a graph obtained from $G$ by adding $|V(G)|$ copies $Q_{1},\cdots ,Q_{|V(G)|}$ of $Q$, where the vertex of degree three in $Q_{i}$ is identified with the $i$th vertex of $G$ and $\mathcal{G}=\{G_{Q}|G$ is a graph$\}$.
\begin{lemma}\label{lem1}
Let $m$ and $k$ be two positive integers such that $m\geq3$ and $k\geq1$, with the conditions that if $m=5,7$, then $k\notin\{2,3,5\}$  and if $m=7$, then $k\neq3$.
Then $$\gamma_{dR}(C_{m,k})\leq\frac{12(m+k)}{11},$$
with equality if and only if $m=5$ and $k=6$. \end{lemma}
\begin{proof}
Since $C_{m,k}$ has a Hamiltonian path, then $\gamma_{dR}(C_{m,k})\leq\gamma_{dR}(P_{m+k})\leq m+k+1$. Now, if $m+k\geq12$, then $m+k+1<\frac{12(m+k)}{11}$ and the result is valid. If $m+k=11$, then $m+k+1=\frac{12(m+k)}{11}$ and the result holds. Finally, if $m+k\leq10$, then by a simple calculation we can see that $\gamma_{dR}(C_{m,k})\leq m+k<\frac{12(m+k)}{11}$.
For equality, by the proof we deduce, if $\gamma_{dR}(C_{m,k})\frac{12(m+k)}{11}$, then it must be  $m+k=11$. In this case $\gamma_{dR}(C_{m,k})=12$. This equality holds if and only if $m=5$ and $k=6$. This completes the proof. \end{proof}

\begin{lemma}\label{lem2}
If $S\neq Q$ is a graph obtained from graphs $C_{m_{1},k_{1}},\cdots ,C_{m_{t},k_{t}}$ and cycles $C_{1},\cdots ,C_{r},\\ C'_{1},\cdots ,C'_{s}$ where $r+s+t\geq2$, by adding a new vertex $z$ and joining $z$ to the leaves of $C_{m_{1},k_{1}},\cdots ,C_{m_{t},k_{t}}$, to exactly one vertex of each $C_{i}$ for $1\leq i\leq r$ and identifying $z$ with one vertex of each $C'_{j}$ for $1\leq j\leq s$, then $\gamma_{dR}(S)\leq\frac{12n(S)}{11}.$\end{lemma}
 \begin{proof}
Let $V(C_{m_{i},k_{i}})=\{x_{1}^{i},\cdots ,x_{m_{i}}^{i},y_{1}^{i},\cdots ,y_{k_{i}}^{i}\}$ where recall that $x_{1}^{i}x_{2}^{i}\cdots x_{m_{i}}^{i}x_{1}^{i}$ is a cycle, $y_{1}^{i}\cdots y_{k_{i}}^{i}$ is a path, and $x_{1}^{i}y_{1}^{i}$ is the edge of $C_{m_{i},k_{i}}$ joining the cycle and the path, for each $1\leq i\leq t$, $V(C_{j})=\{z_{1}^{j},\cdots ,z_{l_{j}}^{j}\}$ for $1\leq j\leq r$ and $V(C'_{j})=\{w_{1}^{j},\cdots ,w_{n_{j}}^{j}\}$ for $1\leq j\leq s$. If $n(S)\in\{5,6,7,8,9,10\}$, then by a simple calculation we can see that $\gamma_{dR}(S)<\frac{12n(S)}{11}$. Let $n(S)\geq 11$. Let $T$ be a tree obtained from $S$ by deleting the edges $$x_{1}^{1}x_{m_{1}}^{1},\cdots ,x_{1}^{t}x_{m_{t}}^{t}, z_{1}^{1}z_{l_{1}}^{1},\cdots ,z_{1}^{r}z_{l_{r}}^{r}, w_{1}^{1}w_{n_{1}}^{1},\cdots ,w_{1}^{s}w_{n_{s}}^{s}$$. If $r+s+t=2$, then $\gamma_{dR}(S)\leq\frac{12n(S)}{11}$ and if $r+s+t\geq3$, then the result follows from Proposition C.
\end{proof}
\begin{lemma}\label{lem3}
Let $H\in \mathcal{H}$ and $u\in V(H)$. If $G$ is a graph obtained from $H$ and $C_{m,k}$ for some integers $m\geq3$ and $k\geq1$ other than $m=5$ and $k=2,3,5$, $m=7$ and $k=3$, by adding the edge $uy_{k}$, then $\gamma_{dR}(G)\leq\frac{12n(G)}{11}.$\end{lemma}
\begin{proof}
Let $f$ be a $\gamma_{dR}(H)$-function and $g$ be a $\gamma_{dR}(C_{m,k})$-function. Then the function $h$ defined by $h(x)=f(x)$ for $x\in V(H)$ and $h(x)=g(x)$ otherwise, is a $DRDF$ of $G$. Lemma \ref{lem2} and the fact $H\in \mathcal{H}$ imply that $\gamma_{dR}(G)\leq\omega(f)+\omega(g)\leq\frac{12n(H)}{11}+\frac{12(m+k)}{11}=\frac{12n(G)}{11}$.
\end{proof}
\begin{lemma}\label{lem4}
Let $H\in\mathcal{H}$ and $u\in V(H)$. If $G$ is a graph obtained from $H$ and a cycle $C_{m}=x_{1},\cdots ,x_{m}x_{1}$ with $m\notin\{5,7\}$, by adding the edge $ux_{1}$, then $\gamma_{dR}(G)\leq\frac{12n(G)}{11}.$\end{lemma}
\begin{proof}
Let $f$ be a $\gamma_{dR}(H)$-function and let $g$ be a $\gamma_{dR}(C_{m})$-function. Then the function $h$ defined by $h(x)=f(x)$ for $x\in V(H)$ and $h(x)=g(x)$ otherwise, is a $DRDF$ of $G$. Now, if $m\leq10$, then since $m\notin\{5,7\}$, $\omega(g)\leq m$. Also, since $H\in\mathcal{H}$ we obtain
$$\gamma_{dR}(G)\leq\omega(f)+\omega(g)\leq\frac{12n(H)}{11}+m=\frac{12(n(G)-m)}{11}+m<\frac{12n(G)}{11}.$$
If $m\geq 11$, then Proposition D and $H\in \mathcal{H}$ imply that
$$\gamma_{dR}(G)\leq\omega(f)+\omega(g)\leq\dfrac{12n(H)}{11}+m+1=\dfrac{12(n(G)-m)}{11}+m+1\leq\dfrac{12n(G)}{11},$$
as desired.
\end{proof}

\begin{lemma}\label{lem5}
Let $H\in \mathcal{H}$ and $u\in V(H)$. If $G$ is a graph obtained from $H$ and a cycle $C_{5}$ and $C_{5,k}$ such that $V(C_{5})=\{z_{1},\cdots ,z_{5}\}$, $V(C_{5,k})=\{x_{1},\cdots ,x_{5},y_{1},\cdots ,y_{k}\}$ where $k\geq1$, $x_{1}$ is adjacent to $y_{1}$ and joining $x_{1}$ to exactly one vertex of $C_{5}$ and joining $u$ to $y_{k}$, then $\gamma_{dR}(G)\leq\frac{12n(G)}{11}.$\end{lemma}
\begin{proof}
Let $f$ be a $\gamma_{dR}(H)$-function and let $g$ be a $\gamma_{dR}(G-H)$-function. Then the function $h$ defined by $h(x)=f(x)$ for $x\in V(H)$ and $h(x)=g(x)$ otherwise, is a $DRDF$ of $G$. Now by a simple calculation we see that $\omega(g)\leq \frac{12(n(G)-n(H))}{11}$. Also, since $H\in\mathcal{H}$ we obtain
$$\gamma_{dR}(G)\leq\omega(f)+\omega(g)\leq\dfrac{12n(H)}{11}+\dfrac{12(G)-n(H))}{11}\leq\dfrac{12n(G)}{11}.$$\end{proof}

Let $\mathcal{F}_{1}$ be the family of all connected multigraphs without loops and with minimum degree at least $3$. Assume that $\mathcal{F}$ is the family of all graphs obtained from some graph
in $\mathcal{F}_{1}$ by subdividing any edge at least once and at most seven except ten times. Note that any graph in $\mathcal{F}$ has order at least $5$. Suppose that A denotes the set of vertices of degree at least $3$ in $G$, and let  $B=V(G)-A$. A path $P$ of $G$ is called maximal if $V(P)\subseteq B$ and each end-vertex of $P$ is adjacent to a vertex of $A$.
 For each $i\geq 1$, let $\mathcal{P}_{i}=\{P |$  $P$ is a maximal path with $|V(P)|=i$ $\}$. Let $\mathcal{P}=\cup_{i\geqslant1}\mathcal{P}_{i}$. Note that $A\cup\bigcup_{P\in \mathcal{P}}V(P)$ is a partition of $V(G)$. For $P\in \mathcal{P}$, let
 $X_{P}=\{ u\in A |$  $u$ is adjacent to an end-vertex of $P\}$. Then $A=\cup_{P\in \mathcal{P}}X_{P}$ and since $G$ is obtained from some multigraph without loops in $\mathcal{F}_{1}$ by subdividing all of its edges at least once, we have $|X_{P}|=2$ for each $P\in\mathcal{P}$. Hence $|A|\geq2$.
\begin{lemma}\label{lem6}
Let $G\in  \mathcal{F}$ and $u$ be a vertex in $A$ such that $deg(u)=max\{deg(x)|$ $x\in A\}$. Let $P_{1},P_{2}\in\mathcal{P}_{4}$, and the end vertices of $P_{1},P_{2}$ have no common vertex except in $u$ and $deg(u)=3$. Then there exists a double Roman dominating function $f$ of $G$ such that  $\omega(f)\leq \frac{12n}{11}$ and $f$ assigns a positive value to every vertex of degree at least $3$.
\end{lemma}
\begin{proof} Let $G\in\mathcal{F}$ be a graph of order $n$. The proof is given by induction on $n$. The result is immediate for $n\leq6$. Suppose	$n\geq7$ and let
the result hold for all graphs in $\mathcal{F}$
 of order less then $n$. Let $G\in\mathcal{F}$ be a graph of order $n\geq 7$. First let $\mathcal{P}_{3}\cup\mathcal{P}_{5}\cup\mathcal{P}_{7}\cup\mathcal{P}_{9}\neq0$. Suppose  $P=x_{1}\cdots x_{2k+1}\in\mathcal{P}_{2k+1}$ where $k\in \{1,2,3,4\}$ and let $X_{P}=\{a_{1},a_{2}\}$
  where $\{ a_{1}x_{1},a_{2}x_{2k	+1}\}\subseteq E(G)$. Assume that $G'= (G-(V(P)-\{ x_{k+1}\}))+\{ a_{1}x_{k+1},  a_{2}x_{k+1}\}$. By the induction hypothesis, there exists a double Roman dominating function $f$
 of $G'$ such that $a_{1},a_{2}\,\in V_{2}\cup\,V_{3}$, and
$\omega(f) \leq\frac{12n'}{11}$. It follows that $f(x_{k'+1})=0$. Then the function $g$, defined by $g(x_{2t})=2$, $g( x_{2t+1}) =0$ where $t\in\{0,1,2,3,4\}$ and $g(x) =f(x)$ otherwise, is a $DRDF$ of $G$ such that $g$ assigns a positive value to every vertex of degree at least 3, and $$\omega(g) =\omega(f)+2k\leq\dfrac{12n'}{11}+2k\leq\dfrac{12n}{11}.$$

 Assume now that $\mathcal{P}=\mathcal{P}_{1}\cup \mathcal{P}_{2k}$ where $k \in\{1,2,3,4,5\}$. Note that $n=|A|+m_{1}+2m_{2}+4m_{4}+6m_{6}+8m_{8}+10m_{10}$ and $m_{1}+m_{2}+m_{4}+m_{6}+m_{8}+m_{10}\geq
 3$ where $m_{t}=|\mathcal{P}_{t}|$ for $t\in\{1,2,4,6,8,10\}$. If $|A|=2$, then let $\mathcal{P}_{4}=\{ v_{1}^{i}v_{2}^{i}v_{3}^{i}v_{4}^{i}\,|\, \ 0\leq i\leq m_{4}\}$, $\mathcal{P}_{6}=\{ w_{1}^{j}w_{2}^{j}w_{3}^{j}w_{4}^{j}w_{5}^{j}w_{6}^{j} \,| \,0\leq j\leq m_{6}\}$, $\mathcal{P}_{8}=\{ y_{1}^{r}y_{2}^{r}y_{3}^{r}y_{4}^{r}y_{5}^{r}y_{6}^{r}y_{7}^{r}y_{8}^{r} \,|\, 0\leq r\leq m_{8}\}$, $\mathcal{P}_{10}=\{ z_{1}^{s}z_{2}^{s}z_{3}^{s}z_{4}^{s}z_{5}^{s}z_{6}^{s}z_{7}^{s}z_{8}^{s}z_{9}^{s}z_{10}^{s} \,|\, 0\leq s\leq m_{10}\}$ and define the function $g:V(G)\rightarrow\{0,1,2,3\}$ by $g(w_{5}^{j})=2$,  $g(x)=g(v_{2}^{i})=g(w_{3}^{j})=g(y_{3}^{r})=g(y_{6}^{r})=g(z_{3}^{s})=g(z_{6}^{s})=g(z_{9}^{s})=3$, for each $x\in A$ and each $0\leq i\leq m_{4}$, $0\leq j\leq m_{6}$, $0\leq r\leq m_{8}$, $0\leq s\leq m_{10}$, and $g(x)=0$ otherwise. Clearly, $g$ is a $DRDF$ of $G$ such that $g$ assigns a positive value to every vertex of degree at least $3$, and $$\omega(g)\leq3|A|+3m_{4}+5m_{6}+6m_{8}+9m_{10}\leq\dfrac{12n}{11}.$$

Henceforth, we assume $|A|\geq3$. We consider the following cases.

$\textbf{Case 1.}$  $u$ is adjacent to two maximal paths $P_{1}\in\mathcal{P}_{2}$ and $P_{2}\in\mathcal{P}_{4}$.\\                                                              Let $P_{1}=x_{1}x_{2}$ and $P_{2}=y_{1}y_{2}y_{3}y_{4}$ and let $\{ux_{1},uy_{1},a_{1}x_{2},a_{2}y_{4}\}\subseteq E(G)$ where $a,b\in A$. Assume that $G'$ is the graph obtained from $G$ by removing the vertices $u,y_{1},y_{2}$ and joining $y_{3}$ to each vertex $z\in N_{G}(u)-\{y_{1}\}$. Clearly, $G'\in\mathcal{F}$. By the induction hypothesis, there exists a double Roman dominating function $f$ of $G'$ such that $f$ assigns a positive value to every vertex of degree at least 3, and $\omega(f)\leq\frac{12(n-3)}{11}$. In particular, $f(y_{3})\geq2$ and $f(a_{1})\geq2$. To double Roman dominate the vertices $x_{1},x_{2}$, we must have $f(y_{3})+f(a_{1})+f(x_{1})+f(x_{2})\geq6$. Without loss of generality, we may assume that $f(y_{3})=f(a_{1})=3$. Define the function $g$ by $g(u)=3$, $g(y_{1})=g(y_{2})=0$ and $g(x)=f(x)$ otherwise. Clearly, $g$ is a $DRDF$ of $G$ such that $g$ assigns a positive value to every vertex of degree at least 3,$$\omega(g)=\omega(f)+3\leq\dfrac{12(n-3)}{11}+3\leq\dfrac{12n}{11}.$$

$\textbf{Case 2.}$  $u$ is adjacent to two paths $p_{1},p_{2}\in\mathcal{P}_{2}$.\\                                                                              Let $P_{1}=x_{1}x_{2}$ and $P_{2}=y_{1}y_{2}$ be two maximal paths in $\mathcal{P}_{2}$ and let $\{ux_{1},uy_{1},ax_{2},by_{2}\}\subseteq E(G)$ where $a,b \in A$. First let $a\neq b$. Assume that $G'$ is
the graph obtained from $G$ by removing the vertices $x_{1},u,y_{1}$ and joining $x_{2}$ to $y_{2}$ and
 joining every vertex $x$ in $N(u)-\{x_{1},y_{1}\}$ to either $a$ or $b$ provided $a$ or $b$ is not adjacent
 to the end-vertex of the maximal path containing $x$. Then by the induction hypothesis, there exists a double Roman dominating
 function $f$ of $G'$ such that f assigns a positive value to every vertex of degree
 at least 3, and $\omega(f)\leq(n-3)+1$. We may assume that $f(a)=f(b)=3$. Define the function $g$ by
 $g(u)=3$, $g(x_{1})=g(y_{1})=0$ and $g(x)=f(x)$ otherwise. Clearly, $g$ is a $DRDF$ of $G$ such that $g$ assigns a
 positive value to every vertex of degree at least 3, and $$\omega(g)=\omega(f)+3\leq\dfrac{12(n-3)}{11}+3\leq\dfrac{12n}{11}.$$
Now let $a=b$. Suppose $G'$ is the graph obtained from $G-x_{2}$ by adding the edge $x_{1}a$. Then by the induction hypothesis,
 there exists a double Roman dominating function $f$ of $G'$ such that $f$ assigns a positive value to every vertex of degree at
 least $3$, and $\omega(f)\leq\frac{12(n-1)}{11}$. We may assume that $f(a)=f(b)=3$. Then the function $g$ defined by $g(x_{2})=0$ and $g(x)=f(x)$ otherwise, is a $DRDF$ of $G$ such that $g$ assigns a positive value to every vertex of degree at least $3$, and $$\omega(g)=\omega(f)\leq\dfrac{12(n-1)}{11}<\dfrac{12n}{11}.$$

$\textbf{Case 3.}$  $u$ is adjacent to a path $P_{1}\in\mathcal{P}_{2k}$ where $k\in \{3,4,5\}$.\\                                                                                Let $P_{1}=x_{1}x_{2}\cdots x_{2k}$ and let $\{ux_{1},ax_{2k}\}\subseteq E(G)$, $a\in A$.
 Assume that $G'=(G-(V(P)-\{x_{1},x_{2},x_{3},x_{4}\}))+ ax_{4}$. Then By the induction hypothesis, there exists a double Roman dominating function $f$ of $G'$ such that $u$, $a \in V_{2}\cup V_{3}$, $\omega(f)\leq\frac{12(n-2)}{11}$. Define the function $g$ by $g(x_{2k-1})=2$ and $g(x_{2k})=0$ and $g(x)=f(x)$ otherwise.
 Clearly, $g$ is a $DRDF$ of $G$ such that $g$ assigns a positive value
 to every vertex of degree at least $3$, and $$\omega(g)=\omega(f)+2\leq\dfrac{12(n-2)}{11}+2<\dfrac{12n}{11}.$$

 $\textbf{Case 4.}$  $u$ is adjacent to two paths $P_{1},P_{2}\in\mathcal{P}_{4}$.\\                                                                                                                             Considering Case 1, we may assume that $u$ is not adjacent to any maximal path in $\mathcal{P}_{2}$. Let $P_{1}=x_{1}x_{2}x_{3}x_{4}$ and $P_{2}=y_{1}y_{2}y_{3}y_{4}$ and let $\{ux_{1},uy_{1},ax_{4},by_{4}\}\subseteq E(G)$ where $a,b\in A$. Considering the following subcases.

$\textbf{Subcase 4.1.}$  $a=b$, $deg(u)=3$.\\                                                                                                                                 Assume that $G'$is the graph obtained by removing the vertices $y_{1},y_{2},y_{3}$ and joining $x_{3}$ to $u$. By the induction hypothesis, there exists a double Roman dominating function $f$ of $G'$ such that $f$ assigns a positive value to every vertex of degree at least $3$, and $\omega(f)\leq\frac{12(n-3)}{11}.$\\
 If we assume that $f(x_{3})=f(u)=3$, $f(x_{1})=f(x_{2})=f(x_{4})=f(y_{4})=0$. Define the function $g$ by $g(y_{3})=3$, $g(y_{1})=g(y_{2})=0$ and $g(x)=f(x)$ otherwise. Clearly, $g$ is a $DRDF$ of $G$ such that $g$ assigns that positive value to every vertex of degree at least $3$, and                                                                                           $$\omega(g)=\omega(f)+3\leq\dfrac{12(n-3)}{11}+3\leq\dfrac{12n}{11}.$$
If we assume that $f(x_{2})=f(a)=3$, $f(x_{1})=f(x_{3})=f(x_{4})=f(y_{4})=0$. Define the function $g$ by $g(y_{2})=3$, $g(y_{1})=g(y_{3})=0$ and $g(x)=f(x)$ otherwise. Clearly, $g$ is a $DRDF$ of $G$ such that $g$ assigns a positive value to every vertex of degree at least $3$, and $$\omega(g)=\omega(f)+3\leq \dfrac{12(n-3)}{11}+3<\dfrac{12n}{11}.$$

$\textbf{Subcase 4.2.}$  $a\neq b$, $deg(u)\geq4$.\\                                                                                                                                 Assume that $G'$ is the graph obtained by removing the vertices $y_{1},y_{2},y_{3}$ and joining $x_{3}$ to $y_{4}$. By the induction hypothesis, there exists a double Roman dominating function $f$ of $G'$ such that $f$ assigns a positive value to every vertex of degree at least $3$, and $\omega(f)\leq\frac{12(n-3)}{11}$. Without loss of generality, we may assume that $f(x_{3})=f(u)=3$. Define the function $g$ by $g(y_{3})=3$, $g(y_{1})=g(y_{2})=0$ and $g(x)=f(x)$ otherwise. Clearly, $g$ is a $DRDF$ of $G$ such that $g$ assigns that a positive value to every vertex of degree at least $3$, and $$\omega(g)=\omega(f)+3\leq\dfrac{12(n-3)}{11}+3\leq\dfrac{12n}{11}.$$

$\textbf{Case 5.}$  $u$ is adjacent to a path $P_{1}\in \mathcal{P}_{4}$ and to two paths $P_{2},P_{3}\in \mathcal{P}_{1}$.
Let $P_{1}=x_{1}x_{2}x_{3}x_{4}$ such that $\{ux_{1},ax_{4}\}\subseteq E(G)$ where $a\in A$. By Case 1,2,3 and 4, we may assume that the other neighbors of $u$ belong to maximal paths in $\mathcal{P}_{1}$. Assume that $G'$ is the graph obtained from $G-\{u,x_{1}\}$ by joining $x_{2}$ to every vertex in $N(u)-\{x_{1}\}$. By the induction hypothesis, there exists a double Roman dominating function $f$ of $G'$ such that $f$ assigns a positive value to every vertex of degree at least $3$, and $\omega(f)\leq\frac{12(n-2)}{11}$. Define the function $g$ by $g(u)=2$, $g(x_{1})=0$ and $g(x)=f(x)$ otherwise. Clearly, $g$ is a $DRDF$ of $G$ such that $g$ assigns that positive value to every vertex of degree at least $3$, and $$\omega(g)=\omega(f)+2\leq \dfrac{12(n-2)}{11}+2<\dfrac{12n}{11}.$$

 Considering the above cases, we assume that $\mathcal{P}=\mathcal{P}_{1}\cup\mathcal{P}_{2}$ and that each vertex in $A$ is
  adjacent to at most one maximal path in $\mathcal{P}_{2}$. Since $deg(a)\geq3$ for each $a\in A$, we deduce that each vertex in $A$ is adjacent to at least two maximal paths in $\mathcal{P}_{1}$. Counting the edges between $A$ and $\cup_{P \in \mathcal{P}_{1}}V(P)$ implies that $|A|\leq m_{1}$. Let $A'=\{u\in A|$ $u$ is adjacent to an end-vertex of a maximal path in $\mathcal{P}_{2}\}$ and $A''=A-A'$. Counting the edges between $A'$ and $\cup_{P \in \mathcal{P}_{2}}V(P)$ yields $|A'|\leq 2m_{2}$. Define the function $g$ by $g(x)=3$ for $x\in A'$, $g(x)=2$ for $x\in A''$ and $g(x)=0$ otherwise. It is to see that $g$ is a $DRDF$ of $G$ that assigns a positive value to every vertex of degree at least $3$ $$\omega(g)\leq 3A'+2A''=2A+A'\leq A+m_{1}+2m_{2}=n<\dfrac{12n}{11}.$$ This completes the proof. \end{proof}

\begin{lemma}\label{lem7}
If $G\in\mathcal{F}$, then there exists a double Roman dominating function $f$ of $G$ such that $\omega(f)\leq \frac{12n}{11}$.
\end{lemma}
\begin{proof}
Let $G\in\mathcal{F}$  be a graph of order $n$. The  proof is given by induction on $n$. The result is immediate for $n\leq6$. Suppose $n\geq7$ and let	the result hold for all graphs in $\mathcal{F}$
 of order less then $n$. Let $G\in\mathcal{F}$ be a graph of order $n\geq 7$. By Lemma \ref{lem6}, we assume that $u$ is adjacent to $P_{1},P_{2}\in\mathcal{P}_{4}$, $deg(u)=3$, $P_{1},P_{2}$ are not adjacent except on $u$.  Now, First suppose that $u$ is adjacent to three  maximal paths $P_{1},P_{2},P_{3}\in\mathcal{P}_{4}$ such that $P_{1}=x_{1}x_{2}x_{3}x_{4}, P_{2}=y_{1}y_{2}y_{3}y_{4}, P_{3}=z_{1}z_{2}z_{3}z_{4}$.
Assume that $G'$ is the graph obtained by removing the vertices $x_{4},x_{3},x_{2},x_{1},u,y_{4},y_{3},y_{2},y_{1},z_{4},z_{3},z_{2},z_{1}$. It is easy to see that $G'\in\mathcal{F}$. By the induction hypothesis and Lemma \ref{lem6},
   there exists a double Roman dominating function $f$ of $G'$
  such that $\omega(f)\leq \frac{12n-13}{11}$. Define the function
  $g$ by $g(u)=g(x_{3})=g(y_{3})=g(z_{3})=3$,
  $g(x_{1})=g(y_{1})=g(z_{1})=g(x_{2})=g(y_{2})=g(z_{2})=g(x_{4})=g(y_{4})=g(z_{4})=0$ and $g(x)=f(x)$ otherwise. Clearly, $g$ is
  a $DRDF$ of $G$ such that
                    $$\omega(g)=\omega(f)+12\leq\dfrac{12(n-13)}{11}+12<\dfrac{12n}{11}.$$
    Henceforth, we may assume that each vertex in $N(u)-\{x_{1},y_{1}\}$ belongs to a in maximal path in $\mathcal{P}_{1}$. Let there exists a path $P_{3}=z\in\mathcal{P}_{1}$ such that $\{uz,zc\}\subseteq E(G)$ where $c\in A-\{u,a\}$. There are two cases.

\textbf{Case 1.} Let $P_{3}$ is not adjacent to $P_{1},P_{2}$ except in $u$.
Assume that $G'$ is the graph obtained by removing
  the vertices $x_{4},x_{3},x_{2},x_{1},u,z,y_{1},y_{2},y_{3},y_{3}$. Clearly, $G'\in\mathcal{F}$. By the induction hypothesis and Lemma \ref{lem6}
  there exists a double Roman dominating function $f$ of $G'$
  such that $\omega(f)\leq \frac{12(n-10)}{11}$. Define the function $g$ by $g(u)=g(x_{3})=g(y_{3})=3$, $g(y_{1})=g(y_{2})=g(y_{4})=g(x_{1})=g(x_{2})=g(x_{4})=g(z)=0$ and $g(x)=f(x)$ otherwise. Clearly, $g$ is a $DRDF$ of such that

               $$\omega(g)=\omega(f)+9\leq\dfrac{12(n-10)}{11}+9<\dfrac{12n}{11}.$$

 \textbf{Case 2.} Let $P_{3}$ is adjacent to $P_{1}$ or $P_{2}$ in $u'$ where $u'\in A$. By Lemma \ref{lem6}, let $u'$ is adjacent to a maximal path $P'=x'_{1}x'_{2}x'_{3}x'_{4}$. Then
Assume that $G'$ is the graph obtained by removing
  the vertices $x'_{4},x'_{3},x'_{2},x'_{1},x_{4},x_{3},x_{2},x_{1},u,u',z,y_{1},y_{2},y_{3},y_{4}$. Clearly, $G'\in\mathcal{F}$. Let $f$ be a $\gamma_{dR}(G')$-function and $g$ be a $\gamma_{dR}(G-G')$-function. Then the function $h$ defined by $h(x)=f(x)$ for $x\in V(G')$ and $h(x)=g(x)$ otherwise, is a $DRDF$ of $G$. By the induction hypothesis $\gamma_{dR}(G')\leq\frac{12(n-15)}{11}$, and by a simple calculation we can see that $\gamma_{dR}(G-G')=15$. Thus
                 $$\omega(g)=\omega(f)+\omega(g)\leq\dfrac{12(n-15)}{11}+15\leq\dfrac{12n}{11}.$$
                This completes the proof. \end{proof}

Let $\mathcal{E}$ be the family of simple graphs $G$ with order $n\geq5$, minimum degree $\delta(G)\geq2$ and with no component isomorphic to $C_{5}$ or $C_{7}$ and $G$ has no an induced subgraph $Q$ or $G=S\neq Q$ or $G$ has an induced subgraph $H$, a maximal path $P'=a_{1},\cdots ,a_{n}$ such that $a_{1}$ is adjacent to exactly one vertex of degree three in $Q$ and $a_{n}$ is adjacent to a vertex of $H$.
\begin{theorem}\label{theorem8}
Let $G\in\mathcal{E}$. Then there exists a double Roman dominating function $f$ of $G$ for which $\omega(f)\leq \frac{12n}{11}$.
There exist some graph in $\mathcal{E}$ for which the equality holds.\end{theorem}
\begin{proof}
We prove the result by induction on $n$. If $n=5$ or $n=6$ and $\Delta\neq3$, then the result holds from Proposition B or D. If $n=6$ and $\Delta=3$, then it is easy to check
that $\gamma_{dR}(G)\leq n<\frac{12n}{11}$. Suppose $n\geq7$ and the result holds for all graph $G\in\mathcal{E}$ of order less then $n$. Let $G\in\mathcal{E}$ be a graph of order $n\geq7$.  Since $\gamma_{dR}(G)\leq\gamma_{dR}(G-e)$ for every
  $e\in E(G)$, we may assume that $|E(G)|$ is as small as possible. If $G$ is disconnected and $G_{1},\cdots ,G_{t}$ are the components of $G$, then it follows from the induction hypothesis that $\gamma_{dR}(G_{i})\leq\frac{12|V(G_{i})|}{11}$ for each $i$ and so
                       $$\gamma_{dR}(G)=\sum_{i=1}^{t}\gamma_{dR}(G_{i})\leq\sum_{i=1}^{t}\dfrac{12|V(G_{i})|}{11}=\dfrac{12n}{11}.$$
                        Thus, we can assume that $G$ is connected. If $\Delta(G)=2$, then $G$ is a path or cycle and the result holds.  Assume that $\Delta(G)\geq3$. Let $A=\{v\in V(G)|$ $deg(v)\geq3\}$ and $B=V(G)-A$. If there are two adjacent vertices $x,y\in A$, then we deduce from the choice of $G$, that $G-xy$ is disconnected and that at least one of the components of $G-xy$ is isomorphic to $C_{5}$ or $C_{7}$.  min$\{deg(x),deg(y)\}=3$. Note that $A=\bigcup_{P\in \mathcal{P}}X_{P}$ and so $A\cup\bigcup_{P\in \mathcal{P}}V(P)$ is a partition of $V(G)$ where $1\leq |X_{P}|\leq2$ for every $P\in\mathcal{P}$. By Lemma \ref{lem6} and lemma \ref{lem7},  we may assume that there exists a maximal path $P$ such that $\delta(G-V(P))\leq1$. This implies that $|X_{P}|=1$ and since $G$ is simple we have $|V(P)|\geq2$. Suppose that $X_{P}=\{a\}$, $P=x_{1}\cdots x_{r}$ and $N_{G}(a)-V(P)=\{b\}$.
 Then there exists the unique maximal path $P'=\,y_{1}\cdots y_{t}$ such that $y_{t}=b$ or $b\in A$. Assume that $y_{1}$ is adjacent to $u$ where $u\in A$.
 For completing the proof there are  some cases.

 \textbf{Case 1.} $|V(P)|=4$, $b\in A$ and $b$ is adjacent to a maximal path $P_l$ where $|V(P_{l})|=n''\equiv0(mod\ 3)$ for some $l$.
 Assume that $l=1$. Let $G'$ be the graph obtained by removing the vertices $V(P_{1})$. Then by the induction hypothesis, there exists a double Roman dominating function $f$ of $G'$ such that $\omega(f)\leq \frac{12n'}{11}$. Let $g$ be a $\gamma_{dR}(P_{1})$-function. Then the function $h$ defined by $h(x)=f(x)$ for $x\in V(G')$ and $h(x)=g(x)$ otherwise, is a $DRDF$ of $G$. Thus $$\gamma_{dR}(G)\leq \omega (f) +\omega (g)\leq  \frac{12n'}{11}+n''\leq \frac{12n}{11}.$$

 \textbf{Case 2.} $|V(P)|=4$, $b\in A$ and $b$ is adjacent to two maximal paths $P_{1}=z_{1}\cdots z_{k}$, $P_{2}=z'_{1}\cdots z'_{k'}$, where $\{z_{k}c,z'_{k'}c\}\subseteq E(G)$ $c\in A$.
  Consider the following subcases.

  \textbf{Subcase 2.1.} $|V(P_{1})|\equiv1(mod\ 6)$, $|V(P_{2})|\equiv2(mod\ 6)$ or $|V(P_{1})|\equiv4(mod\ 6)$, $|V(P_{2})|\equiv5(mod\ 6)$.

\textbf{Subcase 2.1.1.} Assume that $|V(P_{1})|=1$ and $|V(P_{2})|=2$. Let $G'$ be the graph obtained by removing the vertices  $z'_{1},z'_{2}$. Then by the induction hypothesis, there exists a double Roman dominating function $f$ of $G'$ such that $\omega(f)\leq \frac{12n'}{11}$.\\
If $f(b)\in V_{2}\cup V_{3}$ or $f(c)\in V_{2}\cup V_{3}$, then $f$ can be extended to a $DRDF$of $G$ of weight $\omega(f)+2$ and so $\gamma_{dR}(G)\leq\frac{12(n-2)}{11}+2\leq\frac{12n}{11}$. If $f(b)=f(c)=0$, then we may assume, without loos of generality, that $f(z_{1})=2$, $f(a)=f(x_{3})=3$, $f(x_{1})=f(x_{2})=f(x_{4})=0$. The function $g$ defined by $g(z_{1})=g(z'_{1})=g(a)=g(x_{4})=0$, $g(z'_{2})=g(x_{1})=2$, $g(b)=3$ and $g(x)=f(x)$ otherwise, is a $DRDF$ of $G$ such that  $$\gamma_{dR}(G)\leq\frac{12(n-2)}{11}+2\leq\frac{12n}{11}.$$

\textbf{Subcase 2.1.2.} Assume that $|V(P_{1})|\neq1$ or $|V(P_{2})|\neq2$. Let $G'$ be the graph obtained by removing the vertices $z_{1},\cdots ,z_{k-1},z'_{1},\cdots ,z'_{k'-1},a,x_{1},..,x_{4}$ and joining $z_{k}$ to $z'_{k'}$ and  joining every vertex $x$ in $N(b)-\{z_{1},z'_{1}\}$ to $c$. Then by the induction hypothesis, there exists a double Roman dominating function $f$ of $G'$ such that $\omega(f)\leq \frac{12n'}{11}$. We may assume, without loos of generality, that $f(c)=3$. Then the function $g$ defined by $g(x)=f(x)$ for $x\in V(G')$, $g(z_{i})=g(z'_{j})=0$ when $i,j\equiv1,2(mod\ 3)$ and $g(b)=g(z_{i})=g(z'_{j})=3$ when $i,j\equiv0(mod\ 3)$, $g(a)=g(x_{2})=g(x_{4})=0$, $g(x_{2})=2$, $g(x_{3})=3$, is a $DRDF$ of $G$ such that $$\omega(g)=\omega(f)+\dfrac{12(n-n')}{11}\leq\dfrac{12n'}{11}+\dfrac{12(n-n')}{11}\leq\dfrac{12n}{11}.$$
For the following subcases, assume that $G'$ is the graph obtained by removing the vertices $z_{1},\cdots ,z_{k}, z'_{1}, \cdots ,z'_{k'},b$, and joining $c$ to every vertex $x$ in $N(b)-\{z_{1},z'_{1}\}$. Then by the induction hypothesis, there exists a double Roman dominating function $f$ of $G'$ such that $\omega(f)\leq \frac{12n'}{11}$.

 \textbf{Subcase 2.2.} $|V(P_{1})|\equiv2,4(mod6)$, $|V(P_{2})|\equiv2,4(mod6)$ or $|V(P_{1})|\equiv1(mod\ 6)$, $|V(P_{2})|\equiv5(mod\ 6)$.\\
If we assume that $f(c)=f(x_{1})=f(x_{2})=f(x_{4})=0$, $f(a)=f(x_{3})=3$, then define $g$ by $g(z_{i})=g(z'_{j})=0$ where $i,j\equiv1(mod 2)$, $g(z_{i})=g(z'_{j})=2$ where $i,j\equiv0(mod\ 2)$, $g(c)=0$ when $|V(P_{1})|\equiv2,4(mod\ 6)$, $g(c)=2$ otherwise, $g(a)=g(x_{1})=g(x_{3})=0$, $g(x_{2})=3$, $g(b)=g(x_{4})=2$, $g(x)=f(x)$ otherwise.\\
If $f(c)\in V_{2}\cup V_{3}$, then define $g$ by $g(z_{i})=g(z'_{j})=0$ when $i,j\equiv1,2(mod\ 3)$, $g(b)=g(c)=g(z_{i})=g(z'_{j})=3$ when $i,j\equiv0(mod\ 3)$, $g(x)=f(x)$ otherwise.

\textbf{Subcase 2.3.} $|V(P_{1})|\equiv1(mod\ 3)$, $|V(P_{2})|\equiv1(mod\ 3)$.\\
If $f(c)=0$, then define the function  $g$ by $g(b)=g(z_{i})=g(z'_{j})=0$ when $i,j\equiv0,1(mod\ 3)$, $g(c)=g(z_{i})=g(z'_{j})=3$ when $i,j\equiv2(mod\ 3)$, $g(x)=f(x)$ otherwise.\\
If $f(c)\in V_{2}\cup V_{3}$, then define $g$ by $g(z_{i})=g(z'_{j})=0$ when $i,j\equiv1,2(mod\ 3)$, $g(b)=g(z_{i})=g(z'_{j})=3$ when $i,j\equiv2(mod\ 3)$, $g(x)=f(x)$ otherwise.

\textbf{Subcase 2.4.} $|V(P_{1})|\equiv2(mod\ 3)$, $|V(P_{2})|\equiv2(mod\ 3)$.\\
If we assume that $f(c)=f(x_{1})=f(x_{2})=f(x_{4})=0$, $f(a)=f(x_{3})=3$, then define the function $g$ by $g(z_{i})=g(z'_{j})=0$ when $i,j\equiv1(mod\ 2)$, $g(b)=g(c)=g(z_{i})=g(z'_{j})=2$ when $i,j\equiv0(mod\ 2)$, $g(a)=g(x_{1})=g(x_{3})=0, g(x_{2})=3, g(x_{4})=2$, $g(x)=f(x)$ otherwise.\\
If $f(c)\in V_{2}\cup V_{3}$, then define the function $g$ by $g(z_{i})=g(z'_{j})=0$ when $i,j\equiv1,2(mod\ 3)$, $g(b)=g(c)=g(z_{i})=g(z'_{j})=3$ when $i,j\equiv0(mod\ 3)$, $g(x)=f(x)$ otherwise.\\
 Clearly, $g$ is a $DRDF$ of $G$ such that $$\omega(g)=\omega(f)+\dfrac{12(n-n')}{11}\leq\dfrac{12(n')}{11}+\dfrac{12(n-n')}{11}\leq \dfrac{12n}{11}.$$

\textbf{Case 3.} The vertex $b$ is adjacent to maximal paths $P_{i}=z^{i}_{1},\cdots ,z^{i}_{j}$ where $i\geq 2$, $j\geq1$ and $P_{i}$s have no common vertex except in $b$.
Assume that the end vertices of $P_{i}$s are  adjacent to $u_{i}$s, respectively. Consider some subcases as follows.\\

\textbf{Subcase 3.1.} Let there exist $C_{m_{1},k_{1}},\cdots ,C_{m_{t},k_{t}}$ where $V(C_{m_{i},k_{i}})=\{x_{1}^{i},\cdots ,x_{m_{i}}^{i},y_{1}^{i},\cdots ,y_{k_{i}}^{i}\}$ and $b$ be adjacent to $y_{k_{i}}^{i}$ or there exist cycles $C_{1},\cdots ,C_{r}$ where $b$ be adjacent to exactly one vertex of $C_{j}$ that $1\leq j\leq r$. Then by Lemma \ref{lem2}, we assume that \\$G-\{C_{1},\cdots ,C_{r}, C_{m_{1},k_{1}},\cdots ,C_{m_{t},k_{t}},a,b,x_{1},\cdots ,x_{4}\}\neq\emptyset$. Let that $G'$ is the graph obtained by removing the vertices $\{C_{1},\cdots ,C_{r}, C_{m_{1},k_{1}},\cdots ,C_{m_{t},k_{t}},a,b,x_{1},\cdots ,x_{4}\}$, $V(P_{i})$s. Then by the induction hypothesis, there exists a double Roman dominating function $f$ of $G'$ such that $\omega(f)\leq \frac{12n'}{11}$. Let $g$ be a $\gamma_{dR}(G-G')$-function. Then the function $h$ defined by $h(x)=f(x)$ for $x\in V(G')$ and $h(x)=g(x)$ otherwise, is a $DRDF$ of $G$. Proposition E imply that $$\gamma_{dR}(G)\leq \omega (f) +\omega (g)\leq  \dfrac{12n'}{11}+\dfrac{12(n-n')}{11}\leq \dfrac{12n}{11}.$$\\

\textbf{Subcase 3.2.}  Let $|\cup_{i}V(P_{i})|+6\geq11$ or ever $|V(P_{i})|=1$. Then assume that $G'$ is the graph obtained by removing the $G''=G[\bigcup_{i} V(P_{i})\cup\{a,x_{1},\cdots ,x_{4}\}]$. Let $f$ be a $\gamma_{dR}(G')$-function $g$ be a $\gamma_{dR}(G'')$-function. Then the function $h$ defined by $h(x)=f(x)$ for $x\in V(G')$ and $h(x)=g(x)$ otherwise, is a $DRDF$ of $G$. The induction hypothesis and Proposition E imply that $$\gamma_{dR}(G)\leq \omega (f) +\omega (g)\leq  \dfrac{12n'}{11}+\dfrac{12n''}{11}\leq \dfrac{12n}{11}.$$

 \textbf{Subcase 3.3.} $|\cup_{i}V(P_{i})|+6<11$, $u_{1}$ is adjacent to two maximal paths $P'_{1},P'_{2}$ and the end vertices $P'_{1},P'_{2}$ are adjacent to $w$ where $w\in A$. By a similar  argument, using in Case 2, we can see that $$\gamma_{dR}(G)\leq  \dfrac{12n}{11}.$$

\textbf{Subcase 3.4.} $|\cup_{i}V(P_{i})|+6<11$, $u_{1}$ is adjacent to maximal paths $P'_{j}$ where $j>1$ and $P'_{j}$s have no common vertex except in $u_{1}$.\\
Assume that $G'$ is the graph obtained by removing the vertices $V(P'_{j})$s,$V( P_{i})$s,$a,b,x_{1},\cdots ,x_{4}$, special pendant subgraphs attached at $u_{1}$. Let $f$ be a $\gamma_{dR}(G')$-function, $g$ be a $\gamma_{dR}(G-G')$-function. Then the function $h$ defined by $h(x)=f(x)$ for $x\in V(G')$ and $h(x)=g(x)$ otherwise, is a $DRDF$ of $G$. The induction hypothesis and Proposition E imply that$$\gamma_{dR}(G)\leq \omega (f) +\omega (g)\leq  \dfrac{12n'}{11}+\dfrac{12(n-n')}{11}\leq \dfrac{12n}{11}.$$

\textbf{Case 4.} $|V(P)|=4$, $|V(P')|=2$.
  Let $G'$ is the graph obtained by removing the vertices $y_{1},y_{2},a$ and joining $u$ to $x_{1},x_{k}$. Then by the induction hypothesis, there exists a double Roman dominating function $f$ of $G'$ such that $\omega(f)\leq \frac{12n'}{11}.$
 If we assume that $f(u)=f(x_{2})=f(x_{4})=0$, $f(x_{1})=2$, $f(x_{3})=3$. Define the function $g$ by $g(a)=3$, $g(y_{2})=g(x_{1})=0$, $g(y_{1})=2$, $g(x)=f(x)$ otherwise. Clearly, $g$ is a $DRDF$ of $G$ such that $$\omega(g)=\omega(f)+3\leq\frac{12(n-3)}{11}+3< \frac{12n}{11}.$$\\
If we assume $f(u)=f(x_{3})=3$, $f(x_{1})=f(x_{2})=f(x_{4})=0$. Define the function $g$ by $g(a)=3$, $g(y_{1})=g(y_{2})=0$, $g(x)=f(x)$ otherwise. Clearly, $g$ is a $DRDF$ of $G$ such that $$\omega(g)=\omega(f)+3\leq\dfrac{12(n-3)}{11}+3< \dfrac{12n}{11}.$$

\textbf{Case 5.}  $|V(P)|=4$, $|V(P')|=3.$
 Let $G'$ is the graph obtained by removing the vertices $y_{2},y_{3},a$ and joining $y_{1}$ to $x_{1},x_{k}$. Then by the induction hypothesis, there exists a double Roman dominating function $f$ of $G'$ such that $\omega(f)\leq \frac{12n'}{11}.$
 If we assume that $f(y_{1})=f(x_{2})=f(x_{4})=0$, $f(x_{1})=2$, $f(x_{3})=3$. Define the function $g$ by $g(a)=3$, $g(y_{3})=g(x_{1})=0$, $g(y_{2})=2$, $g(x)=f(x)$ otherwise. Clearly, $g$ is a $DRDF$ of $G$ such that $$\omega(g)=\omega(f)+3\leq\dfrac{12(n-3)}{11}+3\leq\dfrac{12n}{11}.$$
If we assume $f(y_{1})=f(x_{3})=3$, $f(x_{1})=f(x_{2})=f(x_{4})=0$. Define the function $g$ by $g(a)=3$, $g(y_{3})=g(y_{2})=0$, $g(x)=f(x)$ otherwise. Clearly, $g$ is a $DRDF$ of $G$ such that $$\omega(g)=\omega(f)+3\leq\dfrac{12(n-3)}{11}+3\leq\dfrac{12n}{11}.$$

\textbf{Case 6.}  $|V(P)|=4$, $|V(P')|=5.$
  Let $G'$ is the graph obtained by removing the vertices $y_{4},y_{5},a$ and joining $y_{3}$ to $x_{1},x_{4}$. Then and by the induction hypothesis, there exists a double Roman dominating function $f$ of $G'$ such that $\omega(f)\leq \frac{12n'}{11}.$
 If we assume that $f(y_{3})=f(x_{2})=f(x_{4})=0$, $f(x_{1})=2$, $f(x_{3})=3$. Define the function $g$ by $g(a)=3$, $g(y_{4})=2$, $g(y_{5})=g(x_{1})=0$, $g(x)=f(x)$ otherwise. Clearly, $g$ is a $DRDF$ of $G$ such that $$\omega(g)=\omega(f)+3\leq\dfrac{12(n-3)}{11}+3\leq \dfrac{12n}{11}.$$
If we assume $f(y_{3})=f(x_{3})=3$, $f(x_{1})=f(x_{2})=f(x_{4})=0$. Define the function $g$ by $g(a)=3$, $g(y_{5})=g(y_{4})=0$, $g(x)=f(x)$ otherwise. Clearly, $g$ is a $DRDF$ of $G$ such that $$\omega(g)=\omega(f)+3\leq\dfrac{12(n-3)}{11}+3\leq\dfrac{12n}{11}.$$

\textbf{Case 7.}  $|V(P)|=6$.
 Assume that $G'$ is the graph obtained from $G$ by removing the vertices $x_{1},x_{2}$ and joining $a$ to $x_{3}$. Clearly, by the induction hypothesis, there exists a double Roman dominating function $f$ of $G'$ such that $\omega(f)\leq \frac{12n'}{11}.$
We may assume that $f(a)=f(x_{4})=3$, $f(x_{3})=f(x_{5})=f(x_{6})=0$. Then the function $g$ defined by $g(x_{1})=0, g(x_{2})=2$, $g(x)=f(x)$ otherwise, is a $DRDF$ of $G$ such that $$\omega(g)=\omega(f)+2\leq\dfrac{12(n-2)}{11}+2\leq \dfrac{12n}{11}.$$
We may assume that $f(a)=f(x_{4})=f(x_{6})=0$, $f(x_{3})=2$, $f(x_{5})=3$. Then the function $g$ defined by $g(x_{2})=0, g(x_{1})=2$, $g(x)=f(x)$ otherwise, is a $DRDF$ of $G$ such that $$\omega(g)=\omega(f)+2\leq\dfrac{12(n-2)}{11}+2\leq \dfrac{12n}{11}.$$\\
According to the pervious Claims and Lemma \ref{lem6}, Lemma \ref{lem7}, we may assume that $G$ has an induced $H$ with $u\in V(H)$ such that $G$ be a graph obtained from $H$ and a cycle $C_{m}=x_{1},\cdots ,x_{m}x_{1}$, by identifying vertices $u$ and $x_{1}$.
Let $z$ denote the vertex resulting by identifying $u$ and $x_{1}$. Then there exists three following case.

\textbf{Subcase 7.1.} If $m\notin\{3,5,6,8,9,11\}$.\\
Let $f$ be a $\gamma_{dR}(H)$-function and let $g$ be a $\gamma_{dR}$-function of the path of order $m-1$ induced by $x_{2}x_{3}\cdots x_{m}$. Then the function $h$ defined by $h(x)=f(x)$ for $x\in V(H)-\{u\}$, $h(z)=f(u)$ and $h(x)=g(x)$ otherwise, is a $DRDF$ of $G$. Using the fact that $m\notin \{3,5,6,8,9,11\}$, a similar argument to that used in the proof of Lemma \ref{lem4} shows that $$\gamma_{dR}(G)\leq\dfrac{12n(G)}{11}.$$

\textbf{Subcase 7.2.} If $m\in\{3,6,8,9,11\}$.\\
Let $z$ is adjacent to maximal path $P_{r}=x_{1}\cdots x_{r}$. Then by the induction hypothesis we have $\gamma_{dR}(G-V(P_{r}))\leq \frac{12(n-r)}{11}$, ever $\gamma_{dR}(G-V(P_{r}))$-function $f$(we assume that $f(z)\in V_{2}\cup V_{3})$ can be extended to a $DRDF$ of $G$ of weight at $\omega(f)+r$(by assigning a 2 to $x_{2i}$ for  $0\leq i\leq \frac{r}{2}$ and a 0 to other vertices  of $P_{r}$ when $r\equiv0(mod 2)$, by assigning a 2 to $x_{2i}$ for $1\leq i\leq \frac{r-2}{2}$ and 3 to $x_{r-1}$ and a 0 to other vertices  of $P_{r}$ when $r\equiv1(mod 2)$, by assigning 0 to $x_{1}$ when $r=1$, $g(x)=f(x)$ for $x\in V(G-P_{r})$.

\textbf{Subcase 7.3.} Let $m\in \{5\}$.\\ Let there exist $C_{m_{1},k_{1}},\cdots ,C_{m_{t},k_{t}}$ where $V(C_{m_{i},k_{i}})=\{x_{1}^{i},\cdots ,x_{m_{i}}^{i},y_{1}^{i},\cdots , y_{k_{i}}^{i}\}$ and  $z$ be adjacent to $y_{k_{i}}^{i}$ or there exists cycles $C_{1},\cdots ,C_{r}$ where $z$ be adjacent to exactly one vertex of $C_{j}$ that $1\leq j\leq r$.\\
By Lemma \ref{lem2}, we may assume that $G-\{C_{1},\cdots ,C_{r}, C_{m_{1},k_{1}},\cdots ,C_{m_{t},k_{t}},a,b,x_{1},\cdots ,x_{4}\}\neq\emptyset$. Let $G'$ be the graph obtained by removing the vertices $\{C_{1},\cdots ,C_{r}, C_{m_{1},k_{1}},\cdots ,C_{m_{t},k_{t}},z,x_{1},\cdots ,x_{4}\}$, $V(P_{i})$s. Then by the induction hypothesis, there exists a double Roman dominating function $f$ of $G'$ such that $\omega(f)\leq \frac{12n'}{11}$. Let $g$ be a $\gamma_{dR}(G-G')$-function. Then the function $h$ defined by $h(x)=f(x)$ for $x\in V(G')$ and $h(x)=g(x)$ otherwise, is a $DRDF$ of $G$. Proposition E imply that $$\gamma_{dR}(G)\leq \omega (f) +\omega (g)\leq  \dfrac{12n'}{11}+\dfrac{12(n-n')}{11}\leq \dfrac{12n}{11}.$$

\textbf{Subcase 7.4.} Let $m\in \{5\}$ and the vertex $z$ is adjacent to  two maximal paths $P_{1}=x_{1},\cdots ,x_{k}$, $P_{2}=y_{1},\cdots ,y_{k'}$, $\{x_{k}c,y_{k'}c\}\subseteq E(G) $ where $c\in A$.\\
Assume that $G'$ is the graph obtained by removing the vertices $x_{1}\cdots  x_{k},y_{1}\cdots  y_{k'}$ and joining $z$ to $c$ and   every vertex $x$ in $N(z)-\{x_{1},y_{1}\}$. Then by the induction hypothesis, there exists a double Roman dominating function $f$ of $G'$ such that $\omega(f)\leq \frac{12(n-(k+k'))}{11}$. Then $f$ can be extended to a $DRDF$ of $G$ of weight at most $\omega(f)+(k+k')$. Thus $$\gamma_{dR}(G)\leq\dfrac{12(n-(k+k'))}{11}+(k+k')\leq\dfrac{12n}{11}.$$

\textbf{Subcase 7.5.} Let $m\in \{5\}$ and the  $z$ is adjacent to maximal paths $P_{i}$ where $i\geq 2$, $P_{i}$s have no common vertex except in $z$.\\
 Let $|C_{5}\cup\bigcup_{i}V(P_{i})|\leq10$. Assume that $G'$ is the graph obtained by removing the vertices $V(P_{l})=x_{1},\cdots ,x_{k}$ where $z$ is adjacent to $P_{l}$. Then by the induction hypothesis, there exists a double Roman dominating function $f$ of $G'$ such that $\omega(f)\leq \frac{12(n-k)}{11}$. we assume that $f(z)=3$. Then $f$ can be extended to a $DRDF$ of $G$ of weight at most $\omega(f)+k$. Thus $$\gamma_{dR}(G)\leq\dfrac{12(n-k)}{11}+k\leq\dfrac{12n}{11}.$$
If $|C_{5}\cup_{i}V(P_{i})|\geq11$. Assume that $G'$ is the graph obtained by removing the vertices $\{z,x_{1},\cdots ,x_{4}\}$ and $V(P_{i})$s. Let $f$ be a $\gamma_{dR}(G')$-function and let $g$ be a $\gamma_{dR}(G-G')$-function. Then the function $h$ defined by $h(x)=f(x)$ for $x\in V(H)$ and $h(x)=g(x)$ otherwise, is a $DRDF$ of $G$. Then the induction hypothesis and Proposition C imply that $$\gamma_{dR}(G)\leq \omega(f)+\omega(G)\leq\dfrac{12n'}{11}+\dfrac{12(n-n')}{11}\leq\dfrac{12n}{11}.$$

For equality, let $H$ be a graph obtained from two cycles of $C_{5}$, adding a new vertex $w$ and joining $w$ to exactly one vertex of each $C_{5}$. For any graph $G$, let $G_{H}$ be the graph obtained from $G$ by adding $|V(G)|$ copies $H_{1},\cdots ,H_{|V(G)|}$ of $H$, identifying $w_{i}$ with the $i$th vertex of $G$. This bound is sharp  for $C_{11}$, $G_{H}$.
$\gamma_{dR}(G_{H})=\frac{12n}{11}$, $\gamma_{dR}(C_{11})=12=\frac{12n}{11}$. This completes the proof. \end{proof}

Finally we prove the conjecture from paper \cite{kkcs}.

\begin{theorem}\label{theorem9}
Let $G$ be a simple graph of order $n$ with minimum degree two   different from $C_{5}$ and $C_{7}$. Then $\gamma_{dR}(G)=\frac{11n}{10}$ or $\frac{11n}{10}$ if and only if $G\in\mathcal{G}$.
\end{theorem}
\begin{proof}
Let $G\in\mathcal{G}$. Then by Proposition D, $\gamma_{dR}(G)=\frac{11n}{10}$.\\
Now let $G\notin\mathcal{G}$. The proof is by induction on $n$. If $n\leq13$, then it is easy to check that  $\gamma_{dR}(G)\neq\frac{11n}{10}$ . Suppose $n\geq14$ and the result holds for all graph $G\notin\mathcal{G}$ of order less than $n$ with minimum degree two different $C_{5}$ and $C_{7}$. Let $G\notin\mathcal{G}$ be a graph of order $n\geq14$ with minimum degree  two  different from $C_{5}$ and $C_{7}$.
If $G\in\mathcal{E}$, then by Theorem \ref{theorem8}, $\gamma_{dR}(G)\neq\frac{11n}{10}$.\\
Now we assume that $G$ has an induced subgraph $Q$ with $u\in V(Q)$ and $u\in A$ such that $u$ is adjacent to $c_{i}\in A$ where $i\geq1$.
Suppose $G'$ is the graph obtained by removing the vertices $V(Q)$ and joining $c_{1}$ to every $c_{i}$ where is not adjacent to $c_{1}$. Let $f$ be a $\gamma_{dR}(G')$-function. Then clearly, $G'\notin\mathcal{G}$ and by induction hypothesis,  $\omega(f)\neq\frac{11(n-10)}{10}$. Then $f$ can be extended to a $DRDF$ of $G$ $\omega(f)+11$ and thus $$\gamma_{dR}(G)=\omega(f)+11\neq\frac{11(n-10)}{10}+11=\frac{11n}{10}.$$
\end{proof}

\end{document}